\documentclass[11pt]{scrartcl}
\pdfoutput=1
\usepackage[utf8x]{inputenc}
\usepackage[T1]{fontenc}
\usepackage{lmodern}

\usepackage[english]{babel}

\usepackage[draft=false,kerning=true]{microtype}

\usepackage{amsmath,amssymb,amsthm,graphicx}
\usepackage{comment}
\usepackage{enumitem}
\usepackage{todonotes}
\usepackage{multirow}
\usepackage{mathtools}
\usepackage{tikz}
\usepackage{hyperref}

\newtheorem{theorem}{Theorem}[section]
\newtheorem{definition}[theorem]{Definition}
\newtheorem{lemma}[theorem]{Lemma}
\newtheorem{corollary}[theorem]{Corollary}
\newtheorem{conjecture}[theorem]{Conjecture}

\newtheorem{proposition}[theorem]{Proposition}

\theoremstyle{remark}
\newtheorem{remark}[theorem]{Remark}
\newtheorem{example}[theorem]{Example}

\newcommand{\cG}{\mathcal{G}}
\newcommand{\cH}{\mathcal{H}}

\newcommand{\Iviz}{I_{\mathrm{viz}}}
\newcommand{\Isos}{I_{\mathrm{sos}}}
\newcommand{\K}{\mathbb{K}}
\newcommand{\Kbar}{\overline{\mathbb{K}}}
\newcommand{\cmN}{\mathit{cmN}}

\newcommand{\Q}{\mathbb{Q}}

\newcommand{\R}{\mathbb{R}}
\newcommand{\sdp}{{\rm SDP}}
\DeclarePairedDelimiter{\parentheses}{(}{)}
\DeclarePairedDelimiter{\set}{\{}{\}}
\newcommand{\variety}[1]{\mathcal{V}(#1)}

\newcommand{\SigSubset}[1]{\sigma_{g,#1}}
\newcommand{\SigPrimeSubset}[1]{\sigma_{g',#1}}
\newcommand{\SigIneq}[1]{\sigma^{I}_{g,#1}}
\newcommand{\SigAll}[1]{\sigma^{A}_{g,#1}}

\title{An Optimization-Based \\
  Sum-of-Squares Approach \\
  to Vizing's Conjecture}

\author{Elisabeth Gaar, Daniel Krenn,\\ Susan Margulies and Angelika Wiegele}
\date{}

\begin{document}

\maketitle

\bgroup
\let\thefootnote\relax\footnotetext{%
  \hspace*{0.5em}This paper is published under a
  Creative Commons Attribution 4.0 International License.
  \raisebox{-0.5ex}{\includegraphics[height=1em]{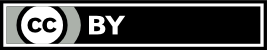}}
  (\url{http://creativecommons.org/licenses/by/4.0/})
  Copyright held by the authors.}
\egroup

\begin{abstract}
  Vizing's conjecture (open since 1968) relates the sizes of dominating
  sets in two graphs to the size of a dominating set in their Cartesian product
  graph. In this paper, we formulate Vizing's conjecture
  itself as a Positivstellensatz existence question. In particular, we
  encode the conjecture as an ideal/polynomial pair such that the
  polynomial is nonnegative if and only if the conjecture is true. We
  demonstrate how to use semidefinite optimization techniques to computationally
  obtain numeric sum-of-squares certificates, and then show how to
  transform these numeric certificates into symbolic certificates
  approving nonnegativity of our polynomial.

  After outlining the theoretical structure of this
  computer-based proof of Vizing's conjecture,
  we present computational and theoretical results.
  In particular,
  we present exact low-degree sparse sum-of-squares certificates for
  particular families of graphs.
\end{abstract}

\section{Introduction}

Sum-of-squares and its relationship to semidefinite
programming is a cutting-edge tool at the forefront of polynomial
optimization \cite{sos-instructions}. Activity in this area has
exploded over the past two decades to span areas as diverse as real
and convex algebraic geometry \cite{sos_geom}, control theory
\cite{sos_control}, proof complexity \cite{sos_proof}, theoretical
computer science \cite{sos_approx} and even quantum computation
\cite{sos_quantum}.  Systems of polynomial equations and other
non-linear models are similarly widely known for their compact and
elegant representations of combinatorial problems. Prior work on
polynomial encodings includes colorings \cite{alontarsi,hillarwindfeldt},
stable sets
\cite{susan_cpc,lovasz1}, matchings \cite{fischer},
and flows \cite{Onn}. In this project, we combine the
modeling strength of systems of polynomial equations with the
computational power of semidefinite programming and devise an
optimization-based framework for a computational proof of an old, open
problem in graph theory, namely Vizing's conjecture.

Vizing's conjecture was first proposed in 1968, and relates the sizes
of minimum dominating sets in graphs $G$ and $H$ to the size of a
minimum dominating set in the Cartesian product graph $G \Box
H$; a precise formulation follows as Conjecture~\ref{conjecture:vizing}.
Prior algebraic work on this conjecture \cite{susan_viz1}
expressed the problem as the union of a certain set of varieties and
thus the intersection of a certain set of ideals. However, algebraic
computational results have remained largely untouched. In this
project, we present an algebraic model of Vizing's conjecture that
equates the validity of the conjecture to the existence of a
Positivstellensatz, or a sum-of-squares certificate of nonnegativity
modulo a carefully constructed ideal.

By exploiting the relationship between the Positivstellensatz and
semidefinite programming, we are able to produce sum-of-squares
certificates for certain classes of graphs where Vizing's conjecture
holds. Thus, not only are we demonstrating an optimization-based
approach towards a computational proof of Vizing's conjecture, but we
are presenting actual minimum degree nonnegativity certificates that
are algebraic proofs of instances of this combinatorial problem.
Although the underlying graphs do not further what is known
about Vizing's conjecture at this time (indeed the
combinatorics of the underlying graphs is fairly trivial), the
construction of these ``combinatorial'' Positivstellens\"atze is an
elegant combination of computation, guesswork and computer algebra that is successfully executed for the first time here.

Our paper is structured as follows. In Section~\ref{sec_back},
we present the necessary background and definitions from graph
theory and commutative algebra. In Section~\ref{sec_poly_I}, we
begin the heart of the paper: we describe the ideal/polynomial pair
that models Vizing's conjecture as a sum-of-squares problem. In
Section~\ref{sec:methodology} we describe our precise process for finding the
sum-of-squares certificates, and in Section~\ref{sec_comp_res} we present our
computational results and the Positivstellens\"atze, i.e., the theorems that
arise. Finally, in Section~\ref{sec_conc}, we summarize our project and
present comments about future work.

\section{Backgrounds and Definitions}\label{sec_back} In this section, we recall all necessary definitions from both graph theory, and polynomial ideals and commutative algebra.

\subsection{Definitions from Graph Theory}

Given a graph $G$ with vertex set~$V(G)$,
a set $D \subseteq V(G)$ is a \emph{dominating set in $G$}
if for each $v \in V(G)\setminus D$, there is a $u \in D$ such that $v$ is
adjacent to $u$ in $G$. The \emph{domination number of $G$}, denoted by
$\gamma(G)$, is the size of a minimum\footnote{Any proper subset of a
  \emph{minimum} dominating set in a graph~$G$ is not a
  dominating set in~$G$.}
dominating set in $G$. The
decision problem of determining whether a given graph has a dominating
set of size $k$ is NP-complete \cite{garey_and_johnson}.

Given graphs $G$ and $H$ with edge sets~$E(G)$ and $E(H)$ respectively,
the Cartesian product graph $G \Box H$ has vertex set
$V(G) \times V(H)$ and edge set
\begin{align*}
  E(G \Box H) = \big\{(gh, g'h') \,:~&\text{$g = g'$ and $(h,h') \in E(H)$, or}\\
  &\text{$h = h'$ and $(g,g') \in E(G)$}\big\}~,
\end{align*}
where $g$, $g' \in V(G)$ and $h$, $h' \in V(H)$.

In 1968, V. Vizing conjectured the following beautiful relationship
between domination numbers and Cartesian product graphs:
\begin{conjecture}[Vizing \cite{vizing}, 1968]\label{conjecture:vizing}
Given graphs $G$ and $H$, then the inequality
\begin{equation*}
  \gamma(G)\,\gamma(H) \leq \gamma(G \Box H)
\end{equation*}
holds.
\end{conjecture}

\tikzset{
  vertex/.style = {circle, fill=gray!20, draw=black, minimum size=8pt, inner sep=0pt},
  edge/.style = {very thick},
  dominating/.style = {fill=gray},
}
\begin{example}\label{ex_prod_graph}
  In this example, we demonstrate the Cartesian product graph of two
  $C_4$ cycle graphs:
    
\begin{center}
  \begin{tikzpicture}[scale=0.7]

    \foreach \y/\sty in {0/, 1/dominating, 2/, 3/dominating}
    \node[vertex, \sty] (G-\y) at (-1.5, \y) {};
    
    \foreach \x/\sty in {0/dominating, 1/, 2/dominating, 3/}
    \node[vertex, \sty] (H-\x) at (\x, 4.5) {};
    
    \foreach \x/\y/\sty in {0/0/, 1/0/, 2/0/dominating, 3/0/,
                            0/1/dominating, 1/1/, 2/1/, 3/1/,
                            0/2/, 1/2/, 2/2/dominating, 3/2/,
                            0/3/dominating, 1/3/, 2/3/, 3/3/}
    \node[vertex, \sty] (GH-\x-\y) at (\x, \y) {};

    \foreach \xa/\xb in {0/1, 1/2, 2/3}
    {
      \foreach \y in {0, 1, 2, 3}
      {
        \draw [edge] (GH-\xa-\y) -- (GH-\xb-\y);
        \draw [edge] (GH-3-\y) to[out=155, in=25] (GH-0-\y);
      }
      \draw [edge] (H-\xa) -- (H-\xb);
    }
    \draw [edge] (H-3) to[out=155, in=25] (H-0);

    \foreach \ya/\yb in {0/1, 1/2, 2/3}
    {
      \foreach \x in {0, 1, 2, 3}
      {
        \draw [edge] (GH-\x-\ya) -- (GH-\x-\yb);
        \draw [edge] (GH-\x-3) to[out=-115, in=115] (GH-\x-0);
      }
      \draw [edge] (G-\ya) -- (G-\yb);
    }
    \draw [edge] (G-3) to[out=-115, in=115] (G-0);

    \node at (-1.5, 3.75) {$G=C_4$};
    \node at (-1.25, 4.5) {$H=C_4$};
    \node at (1.5, -0.75) {$G \Box H$};
  \end{tikzpicture}
\end{center}

In these graphs,~\tikz{\node[vertex, dominating] {};} represents a
vertex in a dominating set, and Vizing's conjecture holds with equality:
$\gamma(G)\,\gamma(H) = 2\cdot 2 = 4 = \gamma(G \Box H)$. However,
observe that some copies of $G$ in $G\Box H$ \emph{do not contain any
  vertices} of the dominating set, i.e., they are dominated entirely
by vertices in other ``layers'' of the graph. This example highlights
the difficulty of Vizing's conjecture.
\hfill $\Box$
\end{example}

\subsection{Historical Notes}
Vizing's conjecture is an active area of research spanning over fifty
years. Early results have focused on proving the conjecture for
certain classes of graphs. For example, in 1979, Barcalkin and
German \cite{viz_bar_ger} proved that Vizing's conjecture holds for
graphs satisfying a certain ``partitioning condition'' on the vertex
set. The idea of a ``partitioning condition'' inspired work for the
next several decades, as Vizing's conjecture was shown to hold on
paths, trees, cycles, chordal graphs, graphs satisfying certain
coloring properties, and graphs with $\gamma(G) \leq 2$. These results
are clearly outlined in the 1998 survey paper by Hartnell and Rall
\cite{viz_dom_bk}. In 2000, Clark and Suen \cite{viz_clark_suen}
showed that $\gamma(G)\,\gamma(H) \leq 2\, \gamma(G \Box H)$, and in 2004,
Sun~\cite{viz_sun} showed that Vizing's conjecture holds on graphs
with $\gamma(G)\leq 3$. Finally, the 2009 survey paper~\cite{viz_survey_2009}
summarizes the work from 1968 to 2008, contains new results,
new proofs of existing results, and comments about minimal
counter-examples.

\subsection{Definitions around Polynomial Ideals}
Our goal is to model Vizing's conjecture as a semidefinite programming
problem. In particular, we will create an ideal/polynomial pair such
that the polynomial is nonnegative over a given variety if
and only if Vizing's conjecture is true. 

In this subsection, we present a brief introduction to polynomial
ideals, and the relationship between nonnegativity and
sum-of-squares. This material is necessary for understanding our
polynomial ideal model of Vizing's conjecture. For a more thorough
introduction to this material see \cite{coxetal} and~\cite{sos-instructions}.

Throughout this section,
let $I$ be an ideal in a polynomial ring~$P=\K[x_1,\ldots,x_n]$
with a field $\K \subseteq \R$.
The \emph{variety of the ideal~$I$} is defined as
the set
\begin{equation*}
  \variety{I} = \{z\in\Kbar^n~:~\text{$f(z)=0$~for all $f \in I$}\}
\end{equation*}
with $\Kbar$ being the algebraic closure of~$\K$.
The variety $\variety{I}$ is called \emph{real}
if $\variety{I} \subseteq \R^n$.

We
say that the ideal~$I$ is \emph{radical}
if whenever $f^m \in I$ for some polynomial~$f\in P$ and integer
$m \geq 1$, then $f \in I$. The
\emph{radical of $I$}, denoted $\sqrt{I}$, is the set
\begin{equation*}
  \sqrt{I} = \{f\in P~:~f^m \in I \text{~for some integer $m \geq 1$}\}.
\end{equation*}
It is easy to
see that an ideal $I$ is radical if and only if $I = \sqrt{I}$.

\begin{lemma} (\cite[Section 3.7.B, pg. 246]{kreuzer})\label{lem_kreuzer_rad}
  Given an ideal $I$ with finite variety~$\variety{I}$,
  if $I$ contains a univariate
  square-free polynomial in each variable, then $I$ is radical.
\end{lemma}
In this case, \emph{square-free} implies that when a polynomial is
decomposed into its unique factorization, there are no repeated
factors.

In particular, Lemma
\ref{lem_kreuzer_rad} implies that ideals containing
$x_i^2 - x_i = x_i(x_i - 1)$ in each variable (i.e., the boolean
ideals) are radical.

We continue with our background by recalling the
necessary notation for sum-of-squares.
\begin{definition}\label{definition:ellsos}
  Let $\ell$ be a nonnegative integer. A polynomial
  $f \in P$ is called \emph{$\ell$-sum-of-squares modulo $I$} (or 
  \emph{$\ell$-sos mod $I$}), if there exist polynomials $s_{1}$, \dots, $s_{d} \in P$ with degrees
  $\deg{s_{i}} \leq \ell$ for all $i \in \set{1, \dots,d}$ and
  \begin{equation*}
    f \equiv \sum_{i=1}^{d} s_{i}^{2} \mod I.
  \end{equation*}
\end{definition}
Algebraic identities like $f = \sum_{i=1}^{d} s_{i}^{2} + g$, $g\in I$,
are often referred to as
\emph{Positivstellensatz certificates of nonnegativity}, and these
identities can be found via semidefinite programming, which is at the
heart of this project.
It is well-known that not all nonnegative polynomials can be
expressed as a sum-of-squares. However, in the particular case when the
ideal is radical and the variety is finite, we can state the following.
\begin{lemma}\label{lem:NonNegVarEquivEllSos} 
  Given a radical ideal $I$ with a finite real variety and a
  polynomial~$f$ with $f(\variety{I}) \subseteq \R$.
  Then $f$ is nonnegative on the variety, i.e,
  $\forall z \in \variety{I} \colon f(z) \geq 0$, if and only if
  there exists a nonnegative integer $\ell$ such that $f$ is $\ell$-sos
  modulo~$I$.
\end{lemma}

\begin{proof}
  Let $f$ be a polynomial that can be expressed as a
  sum-of-squares modulo~$I$. Since all polynomials in the ideal~$I$ vanish
  on the variety by definition and since $\sum_{i=1}^{d} s_{i}^{2}$ is
  clearly positive, $f$ is nonnegative on $\variety{I}$. To
  prove the other direction, we recall a well-known argument included
  here for completeness. Suppose we have a polynomial $f(z) \geq 0$
  for all $z \in \variety{I}$. Suppose further that $\{z_1,\ldots,z_t\}$
  are all points in $\variety{I}$ (recall that the
  variety is finite). We now construct $t$ interpolation polynomials
  (see~\cite{interp})
  such that
    \begin{align*}
    f_i(z) &= \begin{cases}
    1 & z= z_i,\\
    0 & z \neq z_i
    \end{cases}
    \end{align*}
    for all $z \in \variety{I}$.
    Observe that the square of an interpolating polynomial is again an
    interpolating polynomial. Since the ideal is radical, this means that
    $I = I\bigl(\variety{I}\bigr)$ where $I\bigl(\variety{I}\bigr)$
    is the ideal vanishing on $\variety{I}$. In this case, we see
    that the difference polynomial
    \begin{align*}
    \bigg( f(z) - \sum_{i = 1}^t f_i^2(z)f(z_i) \bigg) \in I
    \end{align*}
    since this difference polynomial vanishes on every point in the
    variety. Therefore, if we let $s_i = \sqrt{f(z_i)}  f_i(z)$,
    we then see that
    \begin{equation*}
      f \equiv \sum_{i=1}^d s_i^2 \mod~I.
      \qedhere
    \end{equation*}
\end{proof}

    We observe that the $\ell$ in this case is quite large, since it is
    the degree of the interpolating polynomial $f_i$, which depends on the
    number of points in the variety. However, we will rely on the fact that
    the sum-of-squares representation is not unique, and there may
    exist Positivstellensatz certificates of much lower degree, within
    reach of computation. As we will see in Section
    \ref{sec_comp_res}, this does indeed turn out to be the case.

  \section{Vizing's Conjecture as a Sum-of-Squares Problem} \label{sec_poly_I}
  In this section, we describe Vizing's conjecture as a sum-of-squares
  problem. Towards that end, we will first define ideals associated
  with graphs $G$, $H$ and $G \Box H$, and then finally describe an
  ideal/\allowbreak polynomial pair where the polynomial is nonnegative on the
  variety of the ideal if and only if Vizing's conjecture is true. We begin by
  creating an ideal where the variety of solutions
  corresponds to graphs with a given number of vertices and size of a
  minimum dominating set.

  The notation underlying all of the definitions in this section is as
  follows. 
  Let $n_\cG$ and $k_\cG \leq n_\cG$ be fixed positive integers, and 
  let $\cG$ be the class of graphs on $n_\cG$ vertices with a
  minimum dominating set $D_\cG$ (fixed) of size $k_{\cG}$.
  We then turn the various edges ``on'' or ``off''
  (by controlling a boolean variable~$e_{gg'}$) such that each point in the
  variety corresponds to a \emph{specific} graph $G \in \cG$.

  \begin{definition}\label{def:IG}
    Set $e_\cG = \{e_{gg'} : \set{g,g'} \subseteq V(\cG)\}$.
  The ideal $I_\cG \subseteq P_\cG = \K[e_\cG]$
  is defined by the system of polynomial equations
  \begin{subequations}
    \label{eq:graphs-G-fix}
    \begin{align}
      e_{gg'}^2 - e_{gg'} &= 0 &
      &\text{for $\set{g,g'}\subseteq V(\cG)$,} \label{eq:eij-fix}\\
      \prod_{g' \in D_\cG} (1 - e_{gg'}) &= 0 &
      &\text{for $g \in V(\cG) \setminus D_\cG$,} \label{eq:domset-fix}\\
      \prod_{g' \in V(\cG) \setminus S} \parentheses[\bigg]{\;\sum_{ g \in S}e_{gg'}} &= 0 &
      &\text{for $S \subseteq V(\cG)$ where $|S| = k_\cG-1$.}\label{eq:kcover-fix}
    \end{align}
  \end{subequations}
\end{definition}
\begin{theorem}
  The points in the variety $\variety{I_\cG}$ are in bijection to
  the graphs in $\cG$.
\end{theorem}
\begin{proof}
  Consider any point $z \in \variety{I_\cG}$. Since
  Eqns.~\eqref{eq:eij-fix} turn the edges ``on'' ($e_{gg'} = 1$) or
  ``off'' ($e_{gg'} = 0$), the point $z$ defines a graph~$G$ in
  $n_{\cG}$ vertices. Eqns.~\eqref{eq:domset-fix} iterate over all the
  vertices \emph{inside} the set $D_\cG$, and ensure that for each vertex \emph{outside} the set at least one
  edge from a vertex \emph{inside} the set to this vertex  
  is ``on''. Therefore, $D_\cG$ is a dominating set. Finally,
  Eqns.~\eqref{eq:kcover-fix} iterate over all sets $S$ of size
  $k_\cG-1$ and ensure that at least one vertex \emph{outside}~$S$ is
  \textbf{not} incident on any vertex \emph{inside}~$S$ for any~$S$.
  Therefore, no $S$ of size $k_\cG - 1$ is a dominating
  set. Thus, every point $z \in \variety{I_{\cG}}$ corresponds
  to a graph $G$ on $n_\cG$ vertices with a minimum dominating set of
  size $k_\cG$.
\end{proof}

Similarly, for fixed positive integers $n_\cH$ and $k_\cH \leq n_\cH$, let
$\cH$ be the class of graphs on $n_\cH$
vertices and a minimum dominating set of size $k_\cH$. 
Again, we fix
the dominating set to some $D_\cH$ to
reduce isomorphisms within the variety.
Furthermore let 
the ideal $I_\cH$ be defined on the polynomial
ring~$P_\cH=\K[e_\cH]$ with $e_\cH = \{e_{hh'} : \set{h,h'} \subseteq V(\cH)\}$
such that the
solutions in the variety~$\variety{I_\cH}$ are in bijection to the graphs
in $\cH$.

Next, we define the graph class~$\cG \Box \cH$ and the ideal $I_{\cG \Box \cH}$.
For the above classes~$\cG$ and~$\cH$, the graph class~$\cG \Box \cH$ is
the set of product graphs $G \Box H$ for $G\in\cG$ and $H\in\cH$.
The new variables needed
for the ideal are the variables corresponding to the vertices in the product
graph. Let
$x_{\cG\Box\cH}=\set[\big]{x_{gh} : g \in V(\cG),\, h \in V(\cH)}$ and
set $P_{\cG\Box\cH} = \K[e_\cG \cup e_\cH \cup x_{\cG\Box\cH}]$.

\begin{definition}\label{def:IGH}
  The ideal
  $I_{\cG \Box \cH} \subseteq P_{\cG\Box\cH}$ is defined by the
  system of polynomial equations
  \begin{subequations}
    \label{eq:graphs-GH}
    \begin{align}
      x_{gh}^2 - x_{gh} &= 0, \label{eq:xgh}\\
      \bigl(1-x_{gh}\bigr)
      \parentheses[\bigg]{
        \prod_{\substack{g' \in V(\cG) \\ g' \neq g}}
        \parentheses[\big]{1 - e_{gg'} x_{g'h}}}
      \parentheses[\bigg]{
        \prod_{\substack{h' \in V(\cH) \\ h' \neq h}}
        \parentheses[\big]{1 - e_{hh'} x_{gh'}}} &= 0, \label{eq:xgh_domset}
    \end{align}
  \end{subequations}
 for $g \in V(\cG)$ and $h \in V(\cH)$.
\end{definition}

Observe that we have no restrictions on the edge variables in this
definition. It is only used as a stepping
stone to the final and most important ideal in our polynomial model.

\begin{definition}\label{def:Isos}
  For graph classes~$\cG$ and $\cH$, we set $\Isos$ to be the ideal
  generated by the elements of $I_\cG$, $I_\cH$ and $I_{\cG\Box\cH}$.
\end{definition}
Note that our definition of~$\Isos$ depends on the specific parameters~$n_\cG$, $n_\cH$, $k_\cG$ and $k_\cH$.

\begin{theorem}\label{def:IGHsdp}
  The points in
  the variety $\variety{\Isos}$ are in bijection to the triple of graphs
  whose components are in $\cG$, in $\cH$ and in their corresponding product graph
  with a dominating set of any size.
\end{theorem}
\begin{proof}
  We have already demonstrated that
  $\variety{I_\cG}$, $\variety{I_\cH}$ are in bijection to the
  graphs in $n_\cG$, $n_\cH$ vertices with minimum dominating sets of
  size $k_\cG$, $k_\cH$ respectively. It remains to investigate the
  restrictions placed on the $x_{gh}$ variables, which denote whether
  or not the vertex $gh \in V(\cG \Box \cH)$ appears in the dominating set of the product
  graph. Eqns.~\eqref{eq:xgh} force the vertex variables
  $x_{gh}$ to be ``on'' or ``off'', i.e., the vertex is in the
  dominating set if $x_{gh} = 1$ and is outside the dominating set
  otherwise. Eqns.~\eqref{eq:xgh_domset} force every vertex $gh$ to be
  dominated. It is either in the set itself (i.e., $1 - x_{gh} = 0$),
  or it is adjacent to a vertex in the dominating set via an edge from
  the underlying graph in $\cG$ or the underlying graph in $\cH$. In
  particular, the edge $e_{gg'}$ is ``on'' and the vertex $x_{g'h}$ is
  in the dominating set, or the $e_{hh'}$ is ``on'' and the vertex
  $x_{gh'}$ is in the dominating set. In either of these cases, the
  vertex $x_{gh}$ is dominated. Therefore, the points in the variety
  $\variety{\Isos}$ are in bijection to the graphs in
  $n_\cG$, $n_\cH$ vertices with minimum dominating sets of size
  $k_\cG$, $k_\cH$ respectively, and their corresponding product graph
  with a dominating set of any size.
\end{proof}

Observe that there are no polynomials in $\Isos$ enforcing minimality
on the dominating set in the product graph. This is essential when we
tie all of these ideals and definitions together, and model Vizing's
conjecture as a sum-of-squares problem. In particular, we model
Vizing's conjecture as an ideal/polynomial pair, where the polynomial
must be nonnegative on the variety associated with the ideal if and
only if Vizing's conjecture is true.

\begin{definition}\label{def:fStar}
  Given the graph classes $\cG$ and $\cH$, define
  \begin{equation*}
    f^{\ast}
    = \biggl(\, \sum_{gh \in V(\cG) \times V(\cH)} x_{gh} \biggr) - k_\cG k_\cH.
  \end{equation*}
\end{definition}
\begin{theorem}\label{proposition:VizSDPConj}
  Vizing's conjecture is true if and only if for all values of $n_\cG$, $k_\cG$, $n_\cH$ and $k_\cH$, $f^{\ast}$ is nonnegative on~$\variety{\Isos}$, i.e.,
  \begin{equation*}
    \forall z \in \variety{\Isos} \colon f^{\ast}(z) \geq 0.
  \end{equation*}
\end{theorem}

\begin{proof}
  Assume that Vizing's conjecture is true. Therefore, for all graphs
  $G\in\cG$ and $H\in\cH$, we have $\gamma(G \Box H) \geq \gamma(G)\gamma(H)$. In
  particular, $\gamma(G \Box H) - k_\cG k_\cH \geq 0$, for all
  values of $n_\cG$, $k_\cG$, $n_\cH$ and $k_\cH$. Since $f^{\ast}$ contains
  a sum over all the $x_{gh}$ variables, which represent a dominating
  set in $G\Box H$ of any size, we have $f^{\ast}(z) \geq 0$ 
  for all $z \in \variety{\Isos}$.

  Similarly,
  if $f^{\ast}(z) \geq 0$ for all $z \in \variety{\Isos}$, every
  dominating set in $G \Box H$ has size at least $k_\cG k_\cH$.
  In particular, the minimum dominating set in $G \Box H$ has
  size at least $k_\cG k_\cH$ and Vizing's conjecture is true.
\end{proof}

\begin{corollary}\label{proposition:VizSDPlsos}
  Vizing's conjecture is true if and only if for each $n_\cG$, $k_\cG$, $n_\cH$ and $k_\cH$, there exists an integer $\ell$ such that $f^{\ast}$
  is $\ell$-sos modulo $\Isos$.
\end{corollary}

\begin{proof}
  The ideal $\Isos$ contains the univariate polynomial $x^2 - x$ for
  each variable. Therefore, by Lemma \ref{lem_kreuzer_rad}, $\Isos$ is
  radical. By similar reasoning, $\variety{\Isos}$ is
  finite. Therefore, by Lemma \ref{lem:NonNegVarEquivEllSos}, we know
  that if a polynomial is nonnegative on $\variety{\Isos}$, there
  exists an integer $\ell$ such that the polynomial is $\ell$-sos modulo
  $\Isos$.
\end{proof}

In this section, we have drawn a parallel between Vizing's conjecture
and a sum-of-squares problem. We defined the ideal/polynomial pair
$(\Isos$, $ f^{\ast})$ such that
$f^{\ast}(z) \geq 0$ for all $z \in \variety{\Isos}$ if and only if
Vizing's conjecture is true. In the next section, we describe exactly
how to find these Positivstellensatz certificates of nonnegativity,
or equivalently, these Positivstellensatz certificates that Vizing's conjecture is
true.

\section{Methodology}
\label{sec:methodology}

\subsection{Overview of the Methodology}
In our approach to Vizing's conjecture we ``partition'' the graphs $G$, $H$ and $G \Box H$ by their sizes (number of vertices) $n_\cG$ and $n_\cH$ and by the sizes of
their dominating sets $k_\cG$ and $k_\cH$. 
Note that we aim
for certificates for all partitions as this would prove the conjecture.
However in the following we present
our method which works for a fixed partition (i.e. for fixed values of
$n_\cG$, $k_\cG$, $n_\cH$ and $k_\cH$), and only later relax this and
generalize to parametrized partitions. 

The outline is as follows:

\begin{itemize}
\item Step 1: Model the graph classes as ideals
\item Step 2: \begin{minipage}[t]{0.8\linewidth}
    Formulate Vizing’s conjecture as \\
    sum-of-squares existence question
  \end{minipage}
\item Step 3: Transform to a semidefinite program
\item Step 4: Obtain a numeric certificate
\item Step 5: Guess an exact certificate
\item Step 6: Computationally verify the certificate
\item Step 7: Generalize the certificate
\item Step 8: Prove correctness
\end{itemize}

For fixed values of $n_\cG$, $k_\cG$, $n_\cH$ and $k_\cH$ the first
step is to create the ideal $\Isos$ as described in
Section~\ref{sec_poly_I}, in particular Definition~\ref{def:Isos}. To
summarize, we create the ideal $\Isos$ in a suitable polynomial ring
in such a way that the points in
the variety $\variety{\Isos}$ are in bijection to the triple of graphs
whose components are $\cG$, $\cH$ and their corresponding product graph
with a dominating set of any size. 
In this polynomial ring there is a variable for each possible edge
of $\cG$ and $\cH$ (indicating whether this edge is
present or not in a particular graph) and for each vertex of $\cG \Box \cH$ (indicating whether it is in the
dominating set or not).

The second step is to using the polynomial ring variables mentioned above to 
reformulate Vizing's conjecture: It is true for a fixed partition if a certain polynomial is
nonnegative if evaluated at all points in the variety~$\variety{\Isos}$ of the
constructed ideal. 
For showing that the polynomial is nonnegative, we aim for rewriting it
as a finite sum of squares of polynomials (modulo the ideal~$\Isos$). If we
find such polynomials, then these form a certificate for Vizing's
conjecture for the fixed partition.
To be more precise and as already described in
Section~\ref{sec_poly_I}, Vizing's conjecture is true for this fixed values
of $n_\cG$, $k_\cG$, $n_\cH$ and $k_\cH$ if and only if
$f^{\ast}$ is $\ell$-sos modulo $\Isos$.

In the subsequent Section~\ref{sec:sdp} we describe how to perform step three and to do another
reformulation, namely as a semidefinite program. 
Note that in order of doing so, we need to have
specified $\ell$.
Note also that in order to prepare the
semidefinite program, we need basis polynomials (i.e., special generators)
of the ideals. These
are obtained by computing a Gröbner basis of the ideal;
see~\cite{coxetal}.

The fourth step (Section~\ref{sec:step:numerical-certificate}) is now to solve the
semidefinite program.
If the program is infeasible (i.e., there exists no feasible solution),
we increase $\ell$. 
On the other hand, if the program is 
feasible, then we can construct a numeric
sum-of-squares certificate.
As the underlying system of equations---therefore
the future certificate---is quite large, we iterate
the following tasks: Find a numeric solution to the semidefinite
program, find or guess some structure in the solution, use these new
relations to reduce the size of the semidefinite program, and begin
again with solving. This reduces the solution space and therefore
potentially also the size (number~$d$ of summands) of the certificate and
the number of monomials of the $s_i$ from Definition~\ref{definition:ellsos}.
The procedure above goes hand-in-hand with our next step
(Section~\ref{sec:step:exact-certificate}), namely obtaining (one
might call it guessing) an exact certificate out of the numeric
certificate.

Once we have a candidate for an exact certificate, we can check its
validity computationally by summing up the squares and reducing modulo the ideal; see
our step~six described in Section~\ref{sec:step:check-certificate}. 

We want to
point out, that we still consider Vizing's conjecture for a particular partition of graphs.
However, having such certificates for some partitions, one can go for
generalizing them by introducing parametrized partitions of
graphs. Our seventh step in Section~\ref{sec:step:generalize-certificate} provides more
information.

The final step is to 
prove that the newly obtained, generalized certificate candidate is
indeed a certificate; see Section~\ref{sec:step:generalize-certificate}.

\subsection{Transform to a Semidefinite Program}
\label{sec:sdp}
Semidefinite programming refers to the class of optimization problems 
where a linear function with a symmetric matrix variable is optimized
subject to linear constraints and the constraint that the matrix
variable must be positive semidefinite. A semidefinite program (SDP) can be solved in
polynomial time. In practice the most prominent methods for solving an SDP
efficiently are interior-point algorithms.
We use the solver Mosek~\cite{mosek} within Matlab. For more details on solving SDPs and on interior-point algorithms see~\cite{handbook-1}.

It is possible to check whether a polynomial $f$ is 
$\ell$-sos modulo an ideal with semidefinite programming. We refer
to \cite[pg.~298]{sos-instructions} for detailed information and
examples. We will now present how to do so for our setting only. 

Let us first fix (for example, by computing) a reduced Gröbner basis~$B$ of~$\Isos$ and
fix a nonnegative integer~$\ell$. Denote by $v$ the vector of all
monomials in our polynomial ring~$P$ of degree at most $\ell$ which can not be
reduced\footnote{Algorithmically speaking, we say that a polynomial~$f$ is
  reduced modulo the ideal~$I$ if $f$ is the representative of~$f+I$ which is
  invariant under reduction by a reduced Gröbner basis of the ideal~$I$.}
modulo~$\Isos$ by the Gröbner basis~$B$.
Let $p$ be the length of the vector
$v$. Then $f^{\ast}$ (of Definition~\ref{def:fStar}) is $\ell$-sos modulo $\Isos$ if and only if there
is a positive semidefinite matrix $X\in \R^{p \times p}$
such that $f^{\ast}$ is equal to
\begin{equation*}
  v^{T}Xv
\end{equation*}
when reduced over~$B$. Hence the SDP we end up with optimizes the matrix
variable
$X\in \R^{p \times p}$ subject to linear constraints that
guarantee the above equality. The objective function can be chosen
arbitrarily because any matrix satisfying the constraints is
sufficient for our purpose. More will be said on this later.

If the SDP is feasible, due to the positive
semidefiniteness we can decompose the solution $X$ into $X = S^{T}S$. Then we
define the polynomial~$s_{i}$ by the $i$-th row of $Sz$ and obtain
\begin{equation}\label{eq:sdp-sos-certificate}
 v^{T}Xv = (Sv)^T(Sv) = \sum_i
  s_{i}^2 \equiv f^{\ast} \mod \Isos.
\end{equation}
Note that the last congruence holds due to the constraints in
the SDP. 
Equation~\eqref{eq:sdp-sos-certificate} then certifies that $f^{\ast}$ can
be written as a sum of squares of
the~$s_i$, and hence, $f^{\ast}$ is $\ell$-sos modulo~$\Isos$ according to
Definition~\ref{definition:ellsos}. 

If the SDP is infeasible, we know only that there is no certificate of
degree~$\ell$. We increase~$\ell$ to $\ell+1$, because $f^{\ast}$ could still be
$(\ell+1)$-sos modulo~$\Isos$ or posses a certificate of even higher degree.
However, if no new reduced monomials appear in this increment, then
by Lemma~\ref{lem:NonNegVarEquivEllSos} and Theorem~\ref{proposition:VizSDPConj}
Vizing's conjecture does not hold.

\begin{example} \label{example:3232}
    We consider the graph classes $\cG$ and $\cH$ with $n_\cG = 3$,
    $k_\cG = 2$, $n_\cH=3$ and $k_\cH = 2$. Using
    SageMath~\cite{sagemath} we
    construct the ideal $\Isos$, generated by 32~polynomials in
    15~variables. Again using SageMath, we find a Gröbner basis of
    size~95. 

    First, we check the existence of a 1-sos certificate.
    The vector $v$ for $\ell=1$ has length~$12$, i.e., we set
    up an SDP with a matrix variable $X \in \R^{12 \times 12}$. Imposing
    the necessary constraints to guarantee $\sum_{i} s_{i}^{2} \equiv f^{\ast}
    \mod \Isos$ leads to $67$~linear equality constraints. Interior-point algorithms detect
    infeasibility of this SDP in less than half a second, this implies
    that there is no 1-sos certificate. 

    Setting up the SDP for checking the existence of a 2-sos
    certificate results in a problem with a matrix variable of dimension~$67$
    and~$359$ linear constraints. Interior-point algorithms find a
    solution $X$ of this SDP in $0.72$~seconds, this guarantees the
    existence of a numeric $2$-sos certificate for these graph classes. \hfill $\Box$  

    
    
\end{example}

\subsection{Obtain a Numeric Certificate}
\label{sec:step:numerical-certificate}


As described in Section~\ref{sec:sdp} above, after solving the SDP we decompose
the solution $X$. We do so be computing the eigenvalue decomposition
$X=V^TDV$ and then setting $S=D^{1/2}V$. 
($D$ is the diagonal matrix having the eigenvalues on the main diagonal. Since
$X$ is positive semidefinite, all eigenvalues are nonnegative and we can
compute $D^{1/2}$.)
The matrices $X$, $V$, and $D$ are obtained
through numeric computations, hence there might be entries in $D$ that are rather
close to zero but not considered as zero. 
We may try setting these almost-zero eigenvalues to zero, which
reduces the number of polynomials of the sum-of-squares certificate. 

Furthermore, a zero-column in $S$ means that the corresponding
monomial is not needed in the certificate. Hence, we may try to
compute a certificate where we remove all monomials corresponding to
almost-zero columns. This can decrease the size of the SDP considerably and a smaller size of the matrix and fewer constraints is favorable for solving the SDP.
Of course, if removing these monomials leads to infeasibility of the SDP, then removing these monomials was not correct.


As already mentioned we can choose the objective function arbitrarily.  
Our
experiments show that different objective functions lead to
(significantly) different solutions. Therefore, we carefully choose a
suitable objective function leading to a ``nice'' solution for each
instance.

\begin{example}
\label{example:3232-2}
    We look again at the case we considered in Example~\ref{example:3232}, that is $\cG$ and $\cH$ with $n_\cG = 3$,
    $k_\cG = 2$, $n_\cH=3$ and $k_\cH = 2$, for which we already obtained an optimal solution $X$ and a numeric $2$-sos certificate.
    
    After computing (numerically) the eigenvalue decomposition $X =
    V^TDV$, we set all almost-zero eigenvalues to zero and compute $S
    = D^{1/2}V$, which results in a $12 \times 67$ matrix, i.e.,
    55~eigenvalues are considered as zero.
    In Figure~\ref{fig:heatmap1} a heat map of matrix 
    $S$ is displayed. It seems unattainable to convert this obtained solution
    to an exact certificate (see Section~\ref{sec:step:exact-certificate}),
    so we take a different path.
    \begin{figure}\centering
    \includegraphics[width=0.7\linewidth]{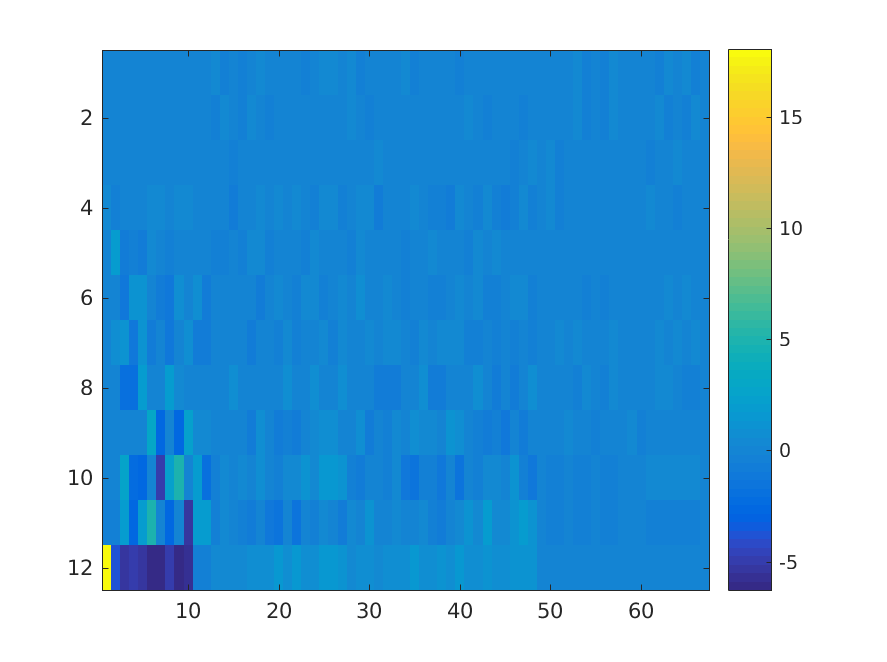}
    \caption{Plotting the entries of matrix $S$. $S_{i,j}$ is the coefficient of the $j$-th monomial in the $i$-th polynomial $s_i$ of the numeric sum-of-squares certificate.}
    \label{fig:heatmap1}
    \end{figure}
    
    Using different objective functions and aiming for a certificate
    where only certain monomials appear can lead to results with a
    clearer structure. If the $i$-th monomial should not be included
    we can set the corresponding $i$-th row and column of $X$ equal to
    zero and obtain another SDP, where we have fewer variables and
    modified constraints. 
    We now try to use only the 19 monomials $1$, $x_{gh}$ and $x_{gh}x_{gh'}$ for all $g \in V(\cG)$ and all $h$, $h'\neq h \in V(\cH)$.
    
    This results in an SDP with a matrix variable of dimension~$19$ and $99$ constraints. The
    SDP can be solved in $0.48$~seconds, and again, we obtain matrix
    $S$ (after setting almost-zero eigenvalues to zero), which now is
    of dimension $4 \times 19$. A heat map is given in Figure~\ref{fig:heatmap2}.
    
    \begin{figure}\centering
    \includegraphics[width=0.7\linewidth]{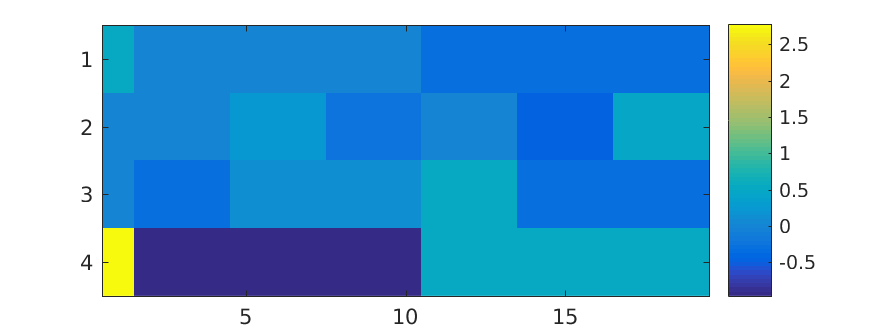}
    \caption{Plotting the entries of matrix $S$ as in Figure~\ref{fig:heatmap1}. The numeric sos certificate consists of 4 (number of rows) polynomials $s_1$, $\dots$, $s_4$ in 19 (number of columns) monomials.}
    \label{fig:heatmap2}
    \end{figure}
    

    As one sees in Figure~\ref{fig:heatmap2}, $S$ has a certain block structure,
    suggesting that
    in each $s_i$ the coefficients of the monomials depend only on the index $g\in V(\cG)$ 
    and there is no dependence on the indices $h \in V(\cH)$.
    Therefore, we aim for a $2$-sos certificate of the form 
    \begin{subequations}\label{eq:possible-certificate-3232-difficult}
    \begin{align}
      s_{i} &=  \nu_i
              + \sum_{g \in V(\cG)} \lambda_{g,i} \biggl( \sum_{h \in V(\cH)}  x_{gh}\biggr)
              +\sum_{g \in V(\cG)} \mu_{g,i} \biggl( \sum_{\set{h, h'} \subseteq V(\cH)}  x_{gh}x_{gh'}\biggr)
              \intertext{for $i \in \set{1, \dots, n_\cG}$ and}
              s_{0} &= \alpha
                      + \beta \biggl( \sum_{g \in V(\cG)} \sum_{h \in V(\cH)} x_{gh} \biggr)
                      + \gamma \sum_{g \in V(\cG)} \biggl( \sum_{\set{h, h'} \subseteq V(\cH)} x_{gh}x_{gh'}\biggr),
    \end{align} 
    \end{subequations}
   where the coefficients $\alpha$, $\beta$, $\gamma$, $\nu_{i}$,
   $\lambda_{g,i}$ and $\mu_{g,i}$ are the entries of $S$. However, we only
   have the numeric values
   \begin{equation*}
     S = \begin{small}\left(\begin{array}{rrrrrrr}
0.535&0.011&0.011&0.011&-0.289&-0.289&-0.289\\
0.000&0.000&0.236&-0.236&-0.001&-0.471&0.472\\
-0.000&-0.272&0.136&0.136&0.544&-0.273&-0.272\\
2.778&-0.962&-0.962&-0.962&0.536&0.536&0.536\\
\end{array}\right)
\end{small}
   \end{equation*}
   at hand and it is not obvious how to guess suitable
   exact numbers from it. 
   In contrast, looking at the values
   \begin{equation*}
     X = \begin{small}\left(\begin{array}{rrrrrrr}
8.000&-2.667 &-2.667&-2.667& 1.333& 1.333&1.333\\
-2.667&1.000 &0.889 & 0.889&-0.667&-0.444&-0.444\\
-2.667&0.889 &1.000 & 0.889&-0.444&-0.667&-0.444\\
-2.667&0.889 &0.889 & 1.000&-0.444&-0.444&-0.667\\
 1.333&-0.667&-0.444&-0.445& 0.667& 0.222&0.222\\
 1.333&-0.444&-0.667&-0.445& 0.222& 0.667&0.222\\
 1.333&-0.444&-0.444&-0.667& 0.222& 0.222&0.667\\
\end{array}\right)
\end{small},
   \end{equation*}
   it seems almost
obvious which simple algebraic numbers the entries of $X$ could be, e.g. $0.667 = 2/3$.
    We will use that in the following section.
   \hfill $\Box$
\end{example}

\subsection{Guess an Exact Certificate}
\label{sec:step:exact-certificate}

We now have a guess for the structure of the certificate, but coefficients that are simple algebraic numbers are hard to
determine from the numbers in $S$. On the other hand, the exact numbers in $X$
seem to be rather obvious so we go back to the relation
$X=S^{T}S$. It implies that if we fix two monomials then the inner product of the vectors of the coefficients of these monomials in all the $s_i$ has to be equal to the corresponding number in $X$.

Setting up a system of equations using all possible inner products, we may
obtain a solution to this system. This solution determines the coefficients in
the certificate (and the certificate might be simplified even further).

\begin{example}
\label{example:3232-3}

    We continue Example~\ref{example:3232}, that is we consider the graph classes $\cG$ and $\cH$ with $n_\cG = 3$,
    $k_\cG = 2$, $n_\cH=3$ and $k_\cH = 2$. 
  
    The exact numbers in $X$ given in Example~\ref{example:3232-2} can be guessed easily.
    In fact, if this guess for $X$ is correct, every choice of $S$ such that
    $S^TS = X$ gives a certificate. Using the relation $S^TS = X$ we set up a
    system of equations on the parameters of~\eqref{eq:possible-certificate-3232-difficult}. 
    To be more precise, let $\lambda_g = (\lambda_{g,i})_{i=1,\dots,n_\cG}$,
    $\mu_g = (\mu_{g,i})_{i=1,\dots,n_\cG}$ and $\nu =
    (\nu_{i})_{i=1,\dots,n_\cG}$. Then we can define the vectors $a =
    \binom{\nu}{\alpha}$, $b_g = \binom{\lambda_g}{\beta}$ and $c_g =
    \binom{\mu_g}{\gamma}$, and $S^TS = X$ (together with the guessed values for $X$) implies that
    \begin{align*}
    \langle a,a \rangle     &= 2(n_\cG - 1)^2, & & \\
    \langle a,b_g \rangle   &= -\tfrac{4}{3}(n_\cG - 1),&    \langle a,c_g \rangle &= \tfrac{2}{3}(n_\cG - 1),\\
    \langle b_g,b_g \rangle &= 1 ,          &    \langle b_g,b_{g'} \rangle &= \tfrac{8}{3},\\
    \langle c_g,c_g \rangle &= \tfrac{6}{9}, &    \langle c_g,c_{g'} \rangle &= \tfrac{2}{9},\\
    \langle b_g,c_g \rangle &= -\tfrac{6}{9},&    \langle b_g,c_{g'} \rangle &= -\tfrac{4}{9}
    \end{align*}
    has to hold for each $g\in V(\cG)$, where $\langle\cdot,\cdot\rangle$
    denotes the standard inner product.
    Under the assumption that our guess for $X$ was correct,
    each solution to this system of equations leads to a valid sum-of-squares
    certificate (\ref{eq:possible-certificate-3232-difficult}a and
    \ref{eq:possible-certificate-3232-difficult}b). 

    We want a sparse certificate and the numeric solution suggests that $\nu_2 = \nu_3 = 0$ holds, so we try to obtain a solution with also $\nu_1 = 0$ (even though the numeric solution does not fit into that setting).
    Using these values, the equations involving the vector $a$ determine the exact values for $\alpha$, $\beta$ and $\gamma$ as $\alpha = \sqrt{2}(n_\cG - 1)$, $\beta = -\frac{2}{3}\sqrt{2}$ and $\gamma = \frac{1}{3}\sqrt{2}$. 
    With that, the system of equations simplifies to
    \begin{align*}
    \langle \lambda_g,\lambda_g \rangle &= \tfrac{1}{9},  & \langle \lambda_g,\lambda_{g'} \rangle &= 0,\\
    \langle \mu_g,\mu_g \rangle         &= \tfrac{4}{9},  & \langle \mu_g,\mu_{g'} \rangle &= 0,\\
    \langle \lambda_g,\mu_g \rangle     &= -\tfrac{2}{9}, & \langle \lambda_g,\mu_{g'} \rangle &= 0.
    \end{align*}
    Calculating $\sum_{i=1}^{n_\cG}(s_i)^2$ we find out that, due to the system of equations, the sum-of-squares simplifies to
    \begin{align*}
    \sum_{i=1}^{n_\cG}(s_i)^2 = \frac{1}{9}\sum_{g \in V(\cG)}\biggl( 
    \biggl( \sum_{h \in V(\cH)}  x_{gh}\biggr) -2 \biggl( \sum_{\set{h, h'} \subseteq V(\cH)}  x_{gh}x_{gh'}\biggr)
    \biggr)^2.
    \end{align*}
    Hence, if~\eqref{eq:possible-certificate-3232-difficult} is a sum-of-squares certificate then also
    \begin{subequations}\label{eq:possible-certificate-3232-easy}
    \begin{align}
      s_{0} &=  \alpha
              + \beta \biggl( \sum_{g \in V(\cG)} \sum_{h \in V(\cH)} x_{gh} \biggr)
              + \gamma \biggl(\sum_{g \in V(\cG)}  \sum_{\set{h, h'} \subseteq V(\cH)} x_{gh}x_{gh'}\biggr),\\    
    s_{g} &= \frac{1}{3}\biggl(  \sum_{h \in V(\cH)}  x_{gh} -  2 \sum_{\set{h, h'} \subseteq V(\cH)}  x_{gh}x_{gh'} \biggr)
    \quad\text{for $ g \in V(\cG)$},
    \end{align}
  \end{subequations}
    where $\alpha = \sqrt{2}(n_\cG - 1)$, $\beta = -\frac{2}{3}\sqrt{2}$ and $\gamma = \frac{1}{3}\sqrt{2}$ is a sum-of-squares certificate.
    \hfill $\Box$
\end{example}

\subsection{Computationally Verify the Certificate}
\label{sec:step:check-certificate}

When a certificate is conjectured, it is straightforward to verify
it computationally via SageMath. To do so, it is necessary to compute the
Gröbner basis of $\Isos$. Observe that at this point, semidefinite programming is no longer needed.

\begin{example}
 We computationally verified (using SageMath) the certificate derived in
 Example~\ref{example:3232-3} for the graph classes $\cG$ and $\cH$ with $n_\cG = 3$,
    $k_\cG = 2$, $n_\cH=3$ and $k_\cH = 2$. 
    \hfill $\Box$
\end{example}

\subsection{Generalize the Certificate and Prove Correctness}
\label{sec:step:generalize-certificate}

In Sections \ref{sec:sdp} to \ref{sec:step:check-certificate}, we presented a methodology for obtaining a sum-of-squares certificate for graph classes $\cG$ and $\cH$ with fixed parameters $n_\cG$,
$k_\cG$, $n_\cH$ and $k_\cH$. Assuming that the previously found pattern generalizes, one can iterate the steps outlined above to obtain certificates for larger classes of graphs.

\begin{example}
    We want to generalize the certificate for the graph classes $\cG$ and $\cH$ with $n_\cG = 3$,
    $k_\cG = 2$, $n_\cH=3$ and $k_\cH = 2$ to the case $k_\cG = n_\cG - 1$, $n_\cH = 3$ and $k_\cH = 2$ for $n_\cG \ge 3$.

    Solving the SDP for the cases $n_\cG=4$ and $n_\cG=5$ again yields nicely structured matrices and in fact, all the calculations done for the case $n_\cG=3$ (which we already wrote down parametrized by $n_\cG$ above) go through.
    Hence, we are able to generalize the sum-of-squares certificate~\eqref{eq:possible-certificate-3232-easy} in the following way.
\end{example}
   
    \begin{theorem} \label{thm_sosCertificate_kHEqnHMinus2}
      For $k_\cG = n_\cG - 1$, $n_\cH = 3$ and $k_\cH = 2$
      Vizing's conjecture is true as the polynomials
        \begin{align*}
          s_{0} &=  \alpha
                  + \beta \biggl( \sum_{g \in V(\cG)} \sum_{h \in V(\cH)} x_{gh} \biggr)
                  + \gamma \biggl(\sum_{g \in V(\cG)}  \sum_{\set{h, h'} \subseteq V(\cH)} x_{gh}x_{gh'}\biggr)   
            \intertext{and}
            s_{g} &= \frac{1}{3}\biggl(  \sum_{h \in V(\cH)}  x_{gh} -  2 \sum_{\set{h, h'} \subseteq V(\cH)}  x_{gh}x_{gh'} \biggr)
                    \quad\text{for $g \in V(\cG)$},
        \end{align*}
        where $\alpha = \sqrt{2}(n_\cG-1)$, $\beta = -\frac{2}{3}\sqrt{2}$
        and $\gamma = \frac{1}{3}\sqrt{2}$,
        are a sum-of-squares certificate with degree~$2$ of $f^{*}$.
    \end{theorem}
    The proof is not included here for
    space considerations. Of course, once having the theorem above, it
    can be verified computationally for particular parameter values, e.g.
    for $k_\cG=4$ and $n_\cG=5$, where the computation of a Gröbner basis
    is feasible.
    \hfill $\Box$

\section{Further Exact Certificates} \label{sec_comp_res}

In the previous section we saw by an example how to use our
machinery combined with clever guessing in order to obtain sum-of-squares
certificates for proving that Vizing's conjecture holds for fixed
values of $n_\cG$, $k_\cG$, $n_\cH$ and $k_\cH$, and how this can be used to
obtain certificates for a less restricted set of parameters. We will use this
section now in order to present further certificates that we found using
the method above and for which we were able to prove
correctness. Again we omit proofs due to space limitations.

\subsection{Certificates for $k_\cG = n_\cG$ and $k_\cH = n_\cH-1$}
The easiest case is the one with $k_\cG = n_\cG$ and $k_\cH = n_\cH-1$. 
We found the following sum-of-squares certificate and therefore know,
that Vizing's conjecture holds in this case.

\begin{theorem}\label{thm:sosCertificaten_kGeqnG_kHeqnHminus1_easy}
  For $k_\cG = n_\cG \geq 2$ and $k_\cH = n_\cH-1 \geq 2$,
  Vizing's conjecture is true as the polynomials
    \begin{equation*}
      s_{g} = \biggl( \sum_{h \in V(\cH)}  x_{gh} \biggr) - k_\cH
      \quad\text{for $g \in V(\cG)$},
    \end{equation*}
    are a $1$-sos certificate of $f^{*}$.
\end{theorem}

Note that the certificate of Theorem~\ref{thm:sosCertificaten_kGeqnG_kHeqnHminus1_easy} has the lowest degree possible and furthermore only uses very particular monomials of degree at most 1.

\subsection{Certificates for $k_\cG = n_\cG$ and $k_\cH = n_\cH-2$}

The next slightly more difficult case is the one for $k_\cG = n_\cG$ and $k_\cH = n_\cH-2$.
Also in this case we were able to find a certificate. 

\begin{theorem}\label{thm:sosCertificaten_kGeqnG_kHeqnHminus2_easy}
  For $k_\cG = n_\cG \geq 2$ and $k_\cH = n_\cH-2 \geq 2$,
  Vizing's conjecture is true as the polynomials
    \begin{equation*}
      s_{g} = \alpha
      + \beta \biggl( \sum_{h \in V(\cH)}   x_{gh} \biggr)
      + \gamma \biggl( \sum_{\set{h, h'} \subseteq V(\cH)} x_{gh}x_{gh'}\biggr)
      \quad\text{for $g \in V(\cG)$},
    \end{equation*}
    with
    \begin{align*}
      \alpha &= (n_\cH - 2)\bigl(n_\cH + \tfrac{1}{2}(n_\cH - 1)\sqrt{2}\bigr) \\
      \beta &= -\bigl( (2n_\cH-3) + (n_\cH - 2)\sqrt{2}\bigr) \\
      \gamma &= 2 + \sqrt{2}
    \end{align*}
    are a $2$-sos certificate of $f^{*}$.
  \end{theorem}
  We want to point out, that this theorem is true
  whenever $\alpha$, $\beta$, $\gamma$ are solutions to the system of equations 
    \begin{subequations}
        \label{eq:systemForkHnHMinus2}
        \begin{align*}
        -(n_\cH-2) &= \alpha^2 + \frac{1}{4}n_\cH(n_\cH-1)(n_\cH-2)(3n_\cH-5)\gamma^2 \hspace{15pt} \nonumber\\
&\phantom{= \alpha^2}\; + n_\cH(n_\cH-1)(n_\cH-2)\beta\gamma,\\
        1 &= \beta^2 + 2\alpha\beta - (n_\cH-1)(n_\cH-2)(2n_\cH-3)\gamma^2 \nonumber\\
&\phantom{= \beta^2}\; - 3(n_\cH-1)(n_\cH-2)\beta\gamma,\\
        0 &= 2\beta^2 + 2\alpha\gamma + (1+3(n_\cH-1)(n_\cH-2))\gamma^2 \nonumber\\
&\phantom{= 2\beta^2}\; + 2(3n_\cH-4)\beta\gamma,
        \end{align*}
    \end{subequations} 
    and that in Theorem~\ref{thm:sosCertificaten_kGeqnG_kHeqnHminus2_easy}
    one particular easy solution is stated.




Note that for all computationally considered instances, the SDP
for $\ell=1$ was infeasible, so for all of those instances there
is no $1$-sos certificate and one really needs monomials of degree
2 in the $s_i$ in order to obtain a certificate.    
Nevertheless, degree 2 is still very low. Furthermore also in this sum-of-squares
certificate only very particular monomials are used; it can be considered
sparse therefore. This is confirmed by the following example.

\begin{example}
If we consider the case $k_\cG= n_\cG = 4$, $n_\cH = 5$ and $k_\cH=3$, there are $432$ monomials of degree at most 2 but the certificate of Theorem~\ref{thm:sosCertificaten_kGeqnG_kHeqnHminus2_easy} uses only $61$ of them.
\hfill $\Box$
\end{example}

\subsection{Certificates for $k_\cG = n_\cG$ and $k_\cH = n_\cH-j$}

When taking a closer look at the certificates in
Theorem~\ref{thm:sosCertificaten_kGeqnG_kHeqnHminus1_easy} and
Theorem~\ref{thm:sosCertificaten_kGeqnG_kHeqnHminus2_easy}, there
seems to be a structure in the certificates we found so far.  In
particular the certificate for the case $k_\cG = n_\cG$ and
$k_\cH = n_\cH-j$ seems to be a $j$-sos certificate.
Hence the following conjecture intuitively seems to be the generalization. 
\begin{conjecture}\label{con:sosCertificaten_kGeqnG_kHeqnHminusi_easy}
    Let $k_\cG = n_\cG$ and $k_\cH = n_\cH-j$ for $j\geq 3$. Then
    \begin{equation*}
    s_{g} = \sum_{q=0}^{j} \alpha_q \biggl(\sum_{\substack{S \subseteq V(\cH)\\|S|=q}}\;\prod_{h \in S} x_{gh} \biggr) \quad\text{for $g \in V(\cG)$},
    \end{equation*}
    where $\alpha_q$ are the solutions to a certain system of polynomial equations,
    are a $j$-sos certificate of $f^{*}$.
\end{conjecture}

\section{Conclusions and Future Work} \label{sec_conc}

In this project, we modeled Vizing's conjecture as an ideal/\allowbreak poly\-nomial
pair such that the polynomial was nonnegative on the variety of a
particularly constructed ideal if and only if Vizing's conjecture was
true. We were able to produce low-degree, sparse Positivstellens\"atze
certificates of nonnegativity for certain classes of graphs using an
innovative collection of techniques ranging from semidefinite
programming to clever guesswork to computer algebra. For example,
Vizing's conjecture with parameters $k_\cG = n_\cG$ and $k_\cH = n_\cH-1$
has a 1-sum-of-squares Positivstellensatz
and with parameters $k_\cG = n_\cG$ and $k_\cH = n_\cH-2$ has a 2-sum-of-squares
Positivstellensatz. We have conjectured a broader combinatorial
pattern based on these certificates, but proving validity is left to
future work. However, at this time, we have indeed proved Vizing's
conjecture for several classes of graphs using sum-of-squares
certificates. Although we have not advanced what is currently known
about Vizing's conjecture, we have introduced a completely new
technique (still to be thoroughly explored) to the literature of
possible approaches.

For future work, we intend to continue pushing the computational
aspect of this project. Additionally, it is very easy to change the
model from a Positivstellensatz certificate to a Hilbert's
Nullstellensatz certificate, and thus change from numeric
semidefinite programming to exact arithmetic linear algebra. This
approach must also be thoroughly investigated. Finally, it would be
very interesting to conjecture a global relationship between the
values of $n_\cG$, $n_\cH$, $k_\cG$ and $k_\cH$, and the degree of the
Positivstellensatz certificate, and perhaps even recast the conjecture
in terms of the theta body hierarchy described in \cite{theta}.

\section*{Acknowledgments} The authors gratefully acknowledge the support of Fulbright Austria (via a Visiting Professorship at AAU Klagenfurt). This project has received funding from the European Union’s Horizon~2020 research and innovation programme under the Marie Sk\l{}odowska-Curie grant agreement No~764759 and the Austrian Science Fund (FWF): I 3199-N31.


\bibliographystyle{amsplain}
{\small\bibliography{papers}}

{\small
  \vspace*{1ex}\noindent
  Elisabeth Gaar,
  \href{mailto:elisabeth.gaar@aau.at}{\url{elisabeth.gaar@aau.at}},
  Alpen-Adria-Universität Klagenfurt,
  Universitätsstraße 65--67, 9020 Klagenfurt, Austria

  \vspace*{1ex}\noindent
  Daniel Krenn,
  \href{mailto:math@danielkrenn.at}{\url{math@danielkrenn.at}},
  Uppsala Universitet,
  Box 480, 75106 Uppsala, Sweden
  
  \vspace*{1ex}\noindent
  Susan Margulies,
  \href{mailto:margulie@usna.edu}{\url{margulie@usna.edu}},
  United States Naval Academy, Annapolis, MD, USA
  
  \vspace*{1ex}\noindent
  Angelika Wiegele,
  \href{mailto:angelika.wiegele@aau.at}{\url{angelika.wiegele@aau.at}},
  Alpen-Adria-Universität Klagenfurt,
  Universitätsstraße 65--67, 9020 Klagenfurt, Austria
}

\end{document}